\documentclass{amsart}

\usepackage{tikz}
\usetikzlibrary{matrix}

\usepackage[mathscr]{euscript}
\usepackage{amssymb}
\usepackage{amsmath}
\usepackage{latexsym}

\makeatletter
\renewcommand\@biblabel[1]{#1.}
\makeatother
\newtheorem{thm}{Theorem}[section]
\newtheorem{lem}[thm]{Lemma}
\newtheorem{cor}[thm]{Corollary}
\newtheorem{prop}[thm]{Proposition}
\newtheorem{df}[thm]{Definition}

\newtheorem{rem}[thm]{Remark}

\newcommand\mf{\mathfrak }
\newcommand{\Fin}{\textrm{Fin}} 
\DeclareMathOperator\Exh{Exh}
\newcommand{\R}{\mathbb{R}}
\newcommand{\cc}{\mf c}
\newcommand{\bb}{\mf b}
\newcommand{\dd}{\mf d}
\newcommand{\pp}{\mf p}
\newcommand{\uu}{\mf u}
\newcommand{\eps}{\varepsilon}
\newcommand{\I}{\mathcal I}
\newcommand{\J}{\mathcal J}
\newcommand{\add}{\tt add}
\newcommand{\non}{\tt non}
\newcommand{\cov}{\tt cov^*}

\title[Ideal weak QN-spaces]{Ideal weak QN-spaces}

\author{Adam Kwela}
\address{Institute of Mathematics, Faculty of Mathematics, Physics and Informatics, University of Gda\'{n}sk, ul.~Wita Stwosza 57, 80-308 Gda\'{n}sk, Poland}
\email{adam.kwela@ug.edu.pl}

\begin{document}
\begin{abstract}
This paper is devoted to studies of $\mathcal{I}$wQN-spaces and some of their cardinal characteristics. 

Recently, \v{S}upina in \cite{Supina} proved that $\mathcal{I}$ is not a weak P-ideal if and only if any topological space is an $\mathcal{I}$QN-space. Moreover, under $\mathfrak{p}=\mathfrak{c}$ he constructed a maximal ideal $\mathcal{I}$ (which is not a weak P-ideal) for which the notions of $\mathcal{I}$QN-space and QN-space do not coincide. In this paper we show that, consistently, there is an ideal $\mathcal{I}$ (which is not a weak P-ideal) for which the notions of $\mathcal{I}$wQN-space and wQN-space do not coincide. This is a partial solution to \cite[Problem 3.7]{BDS}. We also prove that for this ideal the ideal version of Scheepers Conjecture does not hold (this is the first known example of such weak P-ideal). 

We obtain a strictly combinatorial characterization of ${\tt non}(\mathcal{I}\text{wQN-space})$ similar to the one given in \cite{Supina} by \v{S}upina in the case of ${\tt non}(\mathcal{I}\text{QN-space})$. We calculate ${\tt non}(\mathcal{I}\text{QN-space})$ and ${\tt non}(\mathcal{I}\text{wQN-space})$ for some weak P-ideals. Namely, we show that $\mathfrak{b}\leq{\tt non}(\mathcal{I}\text{QN-space})\leq{\tt non}(\mathcal{I}\text{wQN-space})\leq\mathfrak{d}$ for every weak P-ideal $\mathcal{I}$ and that ${\tt non}(\mathcal{I}\text{QN-space})={\tt non}(\mathcal{I}\text{wQN-space})=\mathfrak{b}$ for every $\mathtt{F_\sigma}$ ideal $\mathcal{I}$ as well as for every analytic P-ideal $\mathcal{I}$ generated by an unbounded submeasure (this establishes some new bounds for $\mathfrak{b}(\mathcal{I},\mathcal{I},\textrm{Fin})$ introduced in \cite{Staniszewski}). As a consequence, we obtain some bounds for ${\tt add}(\mathcal{I}\text{QN-space})$. In particular, we get ${\tt add}(\mathcal{I}\text{QN-space})=\mathfrak{b}$ for analytic P-ideals $\I$ generated by unbounded submeasures. 

By a result of Bukovsk\'y, Das and \v{S}upina from \cite{BDS} it is known that in the case of tall ideals $\mathcal{I}$ the notions of $\mathcal{I}$QN-space ($\mathcal{I}$wQN-space) and QN-space (wQN-space) cannot be distinguished. Answering \cite[Problem 3.2]{BDS}, we prove that if $\mathcal{I}$ is a tall ideal and $X$ is a topological space of cardinality less than ${\tt cov^*}(\mathcal{I})$, then $X$ is an $\mathcal{I}$wQN-space if and only if it is a wQN-space.
\end{abstract}


\maketitle

\section{Preliminaries}

The paper is organized as follows. In this Section we introduce basic notions which will be used in further considerations. Section $2$ is devoted to our main results. We investigate $\non(\I\text{QN-space})$ and $\non(\I\text{wQN-space})$ for weak P-ideals. Among other, we show that $\non(\I\text{QN-space})=\non(\I\text{wQN-space})=\bb$ for every $\mathtt{F_\sigma}$ ideal $\I$ as well as for every analytic P-ideal $\I$. We also prove that, consistently, there is an ideal $\I$ (which is not a weak P-ideal) for which the notions of $\I$wQN-space and wQN-space do not coincide. This leads us to a conclusion that the ideal version of Scheepers Conjecture does not hold even for some weak P-ideals. In Section $3$ we show that for any tall ideal $\I$ and a topological space $X$ of cardinality less than $\cov(\I)$ the notions of $\I$wQN-space and wQN-space coincide (this solves \cite[Problem 3.2]{BDS}). Section $4$ is devoted to some remarks concerning additivity of $\I$QN-spaces.

\subsection{Ideals}

A collection $\mathcal{I}$ of subsets of some set $M$ is called an \emph{ideal on $M$} if it is closed under taking finite unions and subsets, contains all finite subsets of $M$ and is a proper subset of $\mathcal{P}(M)$. In this paper we consider only ideals on countable sets. In the theory of ideals a special role is played by the ideal $\Fin=[\omega]^{<\omega}$. 

Ideals $\mathcal{I}$ (on a set $M$) and $\mathcal{J}$ (on a set $N$) are \emph{isomorphic} if there is a bijection $f:N\to M$ such that: 
$$A\in\mathcal{I} \Leftrightarrow f^{-1}[A]\in\mathcal{J}.$$
Results of this paper, although formulated only for ideals on $\omega$, can be generalized for ideals on arbitrary countable sets with the use of isomorphisms.

In our further considerations we will also need the following order on ideals. Let $\I$ and $\J$ be two ideals on $\omega$. We say that \emph{$\I$ is below $\J$ in the Kat\v{e}tov order} and write $\mathcal{I}\leq_{K}\mathcal{J}$ if there is a function $f\colon\omega\to\omega$ such that $A\in\I\Longrightarrow f^{-1}[A]\in\J$ for all $A\subseteq\omega$. If $f$ is a finite-to-one function (i.e., $f^{-1}[\{n\}]$ is finite for all $n\in\omega$), then we say that \emph{$\I$ is below $\J$ in the Kat\v{e}tov-Blass order} and write $\mathcal{I}\leq_{KB}\mathcal{J}$.

A property of ideals can often be expressed by finding a critical ideal (in sense of some order on ideals) with respect to this property (see \cite[Theorem 1.3]{WR}, \cite[Theorem 2]{EUvsSD} or \cite[Theorems 2.1 and 3.3]{Solecki}). This approach is very effective, especially in the context of ideal convergence (see \cite{zReclawem} or \cite{zMarcinem}). One such result, regarding the topic of this paper, will be presented in Subsection \ref{subsection}.

An ideal on a set $M$ is called:
\begin{itemize}
	\item \emph{tall} if every infinite subset of $M$ contains an infinite member of the ideal;
	\item \emph{maximal} if it is maximal with respect to inclusion of ideals on $M$, i.e., there is no other (besides $\I$) ideal on $M$ containing $\I$;
	\item a \emph{P-ideal} if for every $\{A_n:\ n\in\omega\}\subseteq\mathcal{I}$ of $M$ there is $A\in\mathcal{I}$ with $A_n\setminus A$ finite for all $n\in\omega$;
	\item a \emph{weak P-ideal} if for every partition $\{A_n:\ n\in\omega\}\subseteq\mathcal{I}$ of $M$ there is $A\notin\mathcal{I}$ with $A_n\cap A$ finite for all $n\in\omega$.
\end{itemize} 
Clearly, every P-ideal is a weak P-ideal.

$\Fin\otimes\Fin$ is the ideal on $\omega\times\omega$ consisting of all sets $A\subseteq\omega\times\omega$ such that
$$\{n\in\omega:\ A\cap(\{n\}\times\omega)\text{ is infinite}\}\in\Fin.$$
The fact that $\I$ is a weak P-ideal can be expressed equivalently by $\Fin\otimes\Fin\not\leq_{KB}\I$ (for this and other equivalent definitions of this notion, including the ones using different orders on ideals, see \cite[Theorem 3.2]{Supina}).

If $\I$ is an ideal on $\omega$, then we define the ideal $\I\otimes\emptyset$ on $\omega\times\omega$, consisting of all $A\subseteq\omega\times\omega$ such that:
$$\left\{n\in\omega:\ A\cap(\{n\}\times\omega)\neq\emptyset\right\}\in\I.$$

\subsection{Submeasures on $\omega$}
\label{submeasures}

The space $2^X$ of all functions $f:X\to 2$ is equipped with the product topology (each space $2=\left\{0,1\right\}$ carries the discrete topology). We treat $\mathcal{P}(X)$ as the space $2^X$ by identifying subsets of $X$ with their characteristic functions. All topological and descriptive notions in the context of ideals on $X$ will refer to this topology. 

A map $\phi:\mathcal{P}(\omega)\rightarrow[0,\infty]$ is a \emph{submeasure on $\omega$} if $\phi(\emptyset)=0$ and 
$$\phi(A)\leq\phi(A\cup B)\leq\phi(A) + \phi(B)\text{ for all }A,B\subseteq \omega.$$ 
It is \emph{lower semicontinuous} if additionally $\phi(A)=\lim_{n\rightarrow\infty} \phi(A\cap \{0,\ldots,n\})$ for all $A\subseteq\omega$. For a lower semicontinuous submeasure $\phi$ on $\omega$ we define the collections:
$$\Fin(\phi)=\{A\subseteq\omega:\ \phi(A)\text{ is finite}\},$$
$$\Exh(\phi)=\left\{A\subseteq\omega:\ \lim_{n\rightarrow\infty}\phi(A\cap\{n,n+1,\ldots\})=0\right\}.$$

For any lower semicontinuous submeasure $\phi$ on $\omega$ the collection $\Exh(\phi)$ is an $\mathtt{F_{\sigma\delta}}$ P-ideal, while $\Fin(\phi)$ is an $\mathtt{F_\sigma}$ ideal containing $\Exh(\phi)$, provided that $\phi$ is unbounded (see \cite[Lemma 1.2.2]{Farah}). Mazur proved in \cite{Mazur} that every $\mathtt{F_\sigma}$ ideal is equal to $\Fin(\phi)$ for some lower semicontinuous submeasure $\phi$, while in \cite{Solecki} Solecki showed that every analytic P-ideal is equal to $\Exh(\phi)$ for some lower semicontinuous submeasure $\phi$ (see also \cite[Theorem 1.2.5]{Farah}). 

We will be particularly interested in analytic P-ideals generated by unbounded submeasures, i.e., such analytic P-ideals $\I$ that there exists a lower semicontinuous submeasure $\phi$ with $\I=\Exh(\phi)$ and $\phi(\omega)=\infty$. This class contains all $\mathtt{F_\sigma}$ P-ideals, as every such ideal is equal to $\Fin(\phi)=\Exh(\phi)$ for some lower semicontinuous submeasure $\phi$ (see \cite[Theorem 1.2.5]{Farah}). Good examples of $\mathtt{F_\sigma}$ P-ideals are \emph{summable ideals}, i.e., ideals of the form $\I_f=\{A\subseteq\omega:\ \sum_{n\in A}f(n)<\infty\}$ for $f\colon\omega\to\mathbb{R}_+$ such that $\sum_{n\in\omega}f(n)=\infty$ (cf. \cite[Example 1.2.3(c)]{Farah}). It is easy to see that a summable ideal $\I_f$ is tall if and only if $(f(n))$ converges to $0$. There are also analytic P-ideals generated by unbounded submeasures, which are not $\mathtt{F_\sigma}$. A good example is the class of tall density ideals in the sense of Farah, which are not Erd\H{o}s-Ulam ideals (i.e., the class $(\mathcal{Z}4)$ from \cite[Lemma 1.13.9]{Farah}). The class of tall density ideals contains all \emph{simple density ideals}, i.e., ideals of the form $\mathcal{Z}_g=\{A\subseteq\omega:\ \lim_{n\to\infty}\frac{|A\cap n|}{g(n)}=0\}$ for $g\colon\omega\to (0,\infty)$ such that $\lim_{n\to\infty}g(n)=\infty$ and $\left(\frac{n}{g(n)}\right)$ does not converge to $0$ (see \cite[Section 3]{BDFS} for details). By \cite[Proposition 1]{EUvsSD}, a simple density ideal $\mathcal{Z}_g$ is not an Erd\H{o}s-Ulam ideal if and only if the sequence $\left(\frac{n}{g(n)}\right)$ is unbounded (equivalently: $\mathcal{Z}_g$ does not contain the classical ideal $\mathcal{Z}_{\text{id}}$ of sets of asymptotic density zero -- cf. \cite[Theorem 2]{EUvsSD}). In \cite{KPST} it is shown that there are $\cc$ many non-isomorphic simple density ideals which are not Erd\H{o}s-Ulam ideals.

\subsection{Ideal convergence}

Let $\I$ be an ideal on $\omega$. A sequence of reals $(x_n)$ is \emph{$\I$-convergent} to $x\in \R$ if $\{n\in \omega:\ |x_n-x|\geq \varepsilon\}\in \I$ for any $\varepsilon>0$. In this case we write $x_n\xrightarrow{\I} x$. Suppose now that $X$ is a set, $(f_n)\subseteq\mathbb{R}^X$ and $f\in\mathbb{R}^X$. We say that $(f_n)$ is \emph{$\I$-quasi-normally convergent} to $f$ ($f_n\xrightarrow{\I\text{QN}} f$) if there is a sequence $(\varepsilon_n)$ of positive reals with $\varepsilon_n\xrightarrow{\I}0$ such that $\{n\in\omega:\ |f_n(x)-f(x)|\geq \varepsilon_n\}\in\I$ for each $x\in X$. Note that actually in this definition we can require that $(\eps_n)\subseteq(0,1)$.

The above notion generalizes its classical counterpart -- $\Fin$-quasi-normal convergence is called \emph{quasi-normal convergence} or \emph{equal convergence} and has been introduced independently by Bukovsk\'a in \cite{Bukovska} and by Cs{\'a}sz{\'a}r and Laczkovich in \cite{laczkovich-csaszar-equal-convergence-1975}. In \cite{laczkovich-csaszar-equal-convergence-1975} it was shown that quasi-normal convergence is equivalent to $\sigma$-uniform convergence.

The ideal version of quasi-normal convergence has been intensively studied e.g. in \cite{Dascha}, \cite{DDP}, \cite{fil-stan}, \cite{Marcin-Rafal}, \cite{zMarcinem} and \cite{Staniszewski}. 

The next proposition shows that if $\I\neq\J$, then $\I$QN-convergence differs from $\J$QN-convergence. In particular, $\I$QN-convergence differs from QN-convergence for all ideals $\I\neq\Fin$ (this solves the first part of \cite[Problem 3.7]{BDS}).

\begin{prop}
If $\I$ and $\J$ are ideals with $\I\setminus\J\neq\emptyset$, then for any nonempty set $X$ there is a sequence of real-valued functions defined on $X$, which $\I$QN-converges to $0$ but does not $\J$QN-converge to $0$.
\end{prop}

\begin{proof}
Let $A\in\I\setminus\J$. Define a sequence $(f_n)\subseteq\mathbb{R}^X$ by:
$$f_n\left(x\right)=\left\{\begin{array}{ll}
1 & \mbox{\boldmath{if }} n\in A,\\
0 & \mbox{\boldmath{otherwise.}}
\end{array}\right.$$
Then for any sequence $(\eps_n)\subseteq (0,1)$ we have:
$$\left\{n\in\omega:\ |f_n(x)|\geq\eps_n\right\}=A\in\I\setminus\J$$
for every $x\in X$. Hence, $f_n\xrightarrow{\I QN}0$ but $(f_n)$ does not $\J$QN-converge to $0$.
\end{proof}

\subsection{$\I$QN-spaces and $\I$wQN-spaces}
\label{subsection}

For any $B\in[\omega]^\omega$ by $(e_B(n))$ we denote its increasing enumeration, i.e., $e_B:\omega\to B$ is the unique bijection satisfying $e_B(n)<e_B(n+1)$. 

Let $\I$ be an ideal on $\omega$. A topological space $X$ is called:
\begin{itemize}
	\item a \emph{QN-space} if any sequence $(f_n)\subseteq\mathbb{R}^X$ of continuous functions converging to zero quasi-normally converges to zero;
	\item a \emph{wQN-space} if for any sequence $(f_n)\subseteq\mathbb{R}^X$ of continuous functions converging to zero there is a subsequence $(f_{n_k})$ quasi-normally converging to zero, i.e., there is an infinite $B\subseteq\omega$ such that $(f_{e_B(n)})$ quasi-normally converges to zero;
	\item an \emph{$\I$QN-space} if any sequence $(f_n)\subseteq\mathbb{R}^X$ of continuous functions converging to zero $\I$-quasi-normally converges to zero;
	\item an \emph{$\I$wQN-space} if for any sequence $(f_n)\subseteq\mathbb{R}^X$ of continuous functions converging to zero there is a subsequence $(f_{n_k})$ $\I$-quasi-normally converging to zero, i.e., there is an infinite $B\subseteq\omega$ such that $(f_{e_B(n)})$ $\I$-quasi-normally converges to zero.
\end{itemize}
QN-spaces and wQN-spaces were introduced by Bukovsk\'y, Rec\l aw and Repick\'y in \cite{BRR} while their ideal counterparts were defined by Das and Chandra in \cite{Dascha}. Note that Bukovsk\'y, Das and \v{S}upina in \cite{BDS} use a slightly different definition of an $\I$wQN-space -- they allow the sequence $(n_k)$ to be arbitrary, not necessarily increasing. Each $\I$wQN-space (in our sense) fulfills their definition. For more about QN-spaces and wQN-spaces see e.g. \cite{Bukovsky}, \cite{BRR} or \cite{Tsaban}. $\I$QN-spaces and $\I$wQN-spaces are examined e.g. in \cite{Dascha}, \cite{BDS} or \cite{Supina}.

The following diagram presents relations between above notions.
\begin{center}
\begin{tikzpicture}
  \matrix (m) [matrix of math nodes,row sep=3em,column sep=4em,minimum width=2em]
  {
     \text{QN-space} & \text{wQN-space} \\
     \I\text{QN-space} & \I\text{wQN-space} \\};
  \path[-stealth]
    (m-1-1) edge node [left] {} (m-2-1)
            edge node [below] {} (m-1-2)
    (m-2-1.east|-m-2-2) edge node [below] {}
            node [above] {} (m-2-2)
    (m-1-2) edge node [right] {} (m-2-2);
\end{tikzpicture}
\end{center}

Moreover, we have some partial results showing interactions between ideal QN-spaces for different ideals. It is easy to observe that if $\I$ and $\J$ are two ideals on $\omega$ such that $\I\subseteq\J$, then any $\I$QN-space is a $\J$QN-space.

\begin{thm}(\v{S}upina, \cite[Proposition 4.3]{Supina})
\label{Supina1}
Let $\I$ and $\J$ be two ideals on $\omega$ such that $\I\leq_{KB}\J$. Then any $\I$QN-space is a $\J$QN-space.
\end{thm}

\begin{thm}(Bukovsk\'y, Das and \v{S}upina, \cite[Corollary 3.4]{BDS})
\label{BDS}
For a non-tall ideal $\I$ on $\omega$ the notions of $\I$QN-space ($\I$wQN-space) and QN-space (wQN-space) coincide.
\end{thm}

The above result tells us that non-tall ideals are not interesting in the context of ideal QN-spaces and ideal wQN-spaces. Below we present a result showing that this is the case also for ideals which are not weak P-ideals.

\begin{thm}(\v{S}upina, \cite[Theorem 1.4]{Supina})
\label{Supina2}
The following are equivalent for any ideal $\I$ on $\omega$:
\begin{itemize}
	\item[(a)] $\I$ is not a weak P-ideal;
	\item[(b)] every topological space is an $\I$QN-space.
\end{itemize}
\end{thm}

Observe that the above result implies that for non-weak P-ideals $\I$ any topological space is also an $\I$wQN-space.

\subsection{Some cardinal invariants}

Recall the definition of the \emph{pseudointersection number}:
$$\pp=\min\left\{|\mathcal{A}|:\ \mathcal{A}\subseteq[\omega]^\omega\ \wedge\ \forall_{\mathcal{A}_0\in[\mathcal{A}]^{<\omega}}\ \bigcap\mathcal{A}_0\neq\emptyset\ \wedge\ \forall_{\textit{S}\in[\omega]^\omega}\exists_{\textit{A}\in\mathcal{A}}\ |\textit{S}\setminus\textit{A}|=\omega\right\}.$$

\v{S}upina proved that, consistently, the notions of $\I$QN-space and QN-space can be distinguished even for weak P-ideals: if $\pp=\cc$, then there are a maximal ideal $\I$ which is a weak P-ideal and an $\I$QN-space of cardinality $\cc$ which is not a QN-space (\cite[Theorem 1.5]{Supina}). However, the space in this example is a wQN-space. One of the motivations of this paper is to distinguish the notions of wQN-space and $\I$wQN-space in the case of weak P-ideals. This is done in Theorem \ref{ex-IwQN}.

In our further considerations we will also need the following notions. Let $\I$ be an ideal on $\omega$. If $f,g\in\omega^\omega$, then we write $f\leq_\I g$ if $\{n\in\omega:\ f(n)>g(n)\}\in\I$. The cardinals $\bb_\I$ and $\dd_\I$ denote the minimal cardinalities of an unbounded and dominating family in $\omega^\omega$ ordered by $\leq_\I$. We write $\bb_\Fin=\bb$ and $\dd_\Fin=\dd$ for convenience. In it easy to observe that $\bb\leq\bb_\I\leq\dd_\I\leq\dd$ for any ideal $\I$ and that $\bb_\I=\dd_\I$ for any maximal ideal $\I$.

Let $\I$ be a weak P-ideal on $\omega$. Then:
\begin{itemize}
	\item $\non(\I\text{QN-space})$ denotes the minimal cardinality of a perfectly normal topological space which is not an $\I$QN-space;
	\item $\non(\I\text{wQN-space})$ denotes the minimal cardinality of a perfectly normal topological space which is not an $\I$wQN-space.
\end{itemize}

In the case of $\I=\Fin$, by a result of Bukovsk\'y, Rec\l aw and Repick\'y we know the exact values: $\non(\text{QN-space})=\non(\text{wQN-space})=\bb$ (cf. \cite[Corollary 3.2]{BRR}). In \cite[Corollary 6.5]{Supina} it is shown that $\non(\I\text{QN-space})$ has a strictly combinatorial characterization. In Theorem \ref{main} we obtain a similar characterization in the case of $\non(\I\text{wQN-space})$.

\section{Uniformity of $\I$QN-spaces and $\I$wQN-spaces}

\begin{df}
For a weak P-ideal $\I$ on $\omega$ let $\kappa(\I)$ denote the minimal cardinality of a family $\mathcal{A}\subseteq\Fin^\omega$ with the property that for every partition $(B_n)_{n\in\omega\cup\{-1\}}$ of $\omega$ satisfying $e_B^{-1}[B_n]\in\I$ for all $n\in\omega$, where $B=\bigcup_{n\in\omega}B_n$, there is $(A_n)\in\mathcal{A}$ such that 
$$e_B^{-1}\left[\bigcup_{n\in\omega}A_n\cap B_n\right]\notin\I.$$
\end{df}

\begin{rem}
Notice that $\kappa(\I)$ can be defined in a slightly different (and perhaps less complicated) way. For an ideal $\I$ on $\omega$ denote by $\mathcal{P}_\I$ the family of all partitions of $\omega$ into sets belonging to $\I$. Then 
$$\kappa(\I)=\min\left\{|\mathcal{A}|:\ \mathcal{A}\subseteq\Fin^\omega\ \wedge\ \forall_{B\in[\omega]^\omega}\ \forall_{(D_n)\in\mathcal{P}_\I}\ \exists_{(A_n)\in\mathcal{A}} \bigcup_{n\in\omega}e_B^{-1}[A_n]\cap D_n\notin\I \right\}.$$
Indeed, given $(B_n)_{n\in\omega\cup\{-1\}}$ such as above, it suffices to put $B=\bigcup_{n\in\omega}B_n$ and $D_n=e_B^{-1}[B_n]$ for each $n\in\omega$. On the other hand, given $B\in[\omega]^\omega$ and $(D_n)\in\mathcal{P}_\I$, put $B_{-1}=\omega\setminus B$ and $B_n=e_B[D_n]$ for each $n\in\omega$ to obtain the required partition.
\end{rem}

\begin{thm}
\label{nonIwQN}
Let $\I$ be a weak P-ideal on $\omega$. The following are equivalent for any set $X$:
\begin{itemize}
	\item[(a)] $|X|<\kappa(\I)$;
	\item[(b)] for any sequence of real-valued functions defined on X which converges to some $f\in\mathbb{R}^X$ one can find its subsequence $\I$QN-converging to $f$;
	\item[(c)] $X$ with the discrete topology is an $\I$wQN-space.
\end{itemize}
\end{thm}

\begin{proof}
This proof is only a slight modification of the proof of \cite[Theorem 5.1]{Marcin-Rafal}.

The implication (b)$\Rightarrow$(c) is obvious. We will prove (a)$\Rightarrow$(b) and (c)$\Rightarrow$(a).

{\bf (a)$\Rightarrow$(b): }Suppose that $|X|<\kappa(\I)$ and $(f_n)\subseteq\mathbb{R}^X$ converges to some $f\in\mathbb{R}^X$. Define
$$A^x_k=\left\{n\in\omega:\ |f_n(x)-f(x)|\geq\frac{1}{k+1}\right\}\in\Fin$$
for each $x\in X$ and $k\in\omega$. Let $(B_n)_{n\in\omega\cup\{-1\}}$ be the partition of $\omega$ which exists by the definition of $\kappa(\I)$ and denote $B=\bigcup_{n\in\omega}B_n$. Define an $\I$-converging to $0$ sequence $(\eps_k)\subseteq(0,1]$ by:
$$\eps_k=\frac{1}{n+1}\ \Longleftrightarrow\ k\in e_B^{-1}[B_n]\in\I.$$

Fix any $x\in X$ and observe that:
$$\left\{k\in\omega:\ |f_{e_B(k)}(x)-f(x)|\geq\eps_k\right\}=\bigcup_{n\in\omega}\left\{k\in e_B^{-1}[B_n]:\ |f_{e_B(k)}(x)-f(x)|\geq\eps_k\right\}=$$
$$=\bigcup_{n\in\omega}\left\{k\in e_B^{-1}[B_n]:\ |f_{e_B(k)}(x)-f(x)|\geq\frac{1}{n+1}\right\}=e_B^{-1}\left[\bigcup_{n\in\omega}A^x_n\cap B_n\right]\in\I.$$

{\bf (c)$\Rightarrow$(a): }Suppose to the contrary that $|X|\geq\kappa(\I)$. Let $\phi\colon\kappa(\I)\to X$ be an injection. Suppose that $\mathcal{A}=\{(A_n^\alpha):\ \alpha<\kappa(\I)\}$ is the family from the definition of $\kappa(\I)$. Define on $X$ real-valued functions $(f_n)$ by:
$$f_n\left(x\right)=\left\{\begin{array}{ll}
\frac{1}{k+1} & \mbox{\boldmath{if }} n\in A^\alpha_k\setminus\bigcup_{m<k}A^\alpha_m\mbox{\boldmath{ and }}x=\phi(\alpha)\mbox{\boldmath{ for some }}\alpha<\kappa(\I),\\
0 & \mbox{\boldmath{otherwise,}}\\
\end{array}\right.$$
for each $n\in\omega$. Then $f_n$ converges to $0$, so by our assumption it has a subsequence $(f_{n_k})$ which $\I$QN-converges to $0$. Let $(\eps_k)\subseteq(0,1)$ be the $\I$-converging to $0$ sequence witnessing it. 

Define: 
$$B=\{n_k:\ k\in\omega\},\quad B_{-1}=\omega\setminus B\quad\text{ and }$$
$$B_n=\left\{n_k:\ \frac{1}{n+2}\leq\eps_k<\frac{1}{n+1}\right\}\quad\text{ for all }\quad n\in\omega$$
(observe that $e_B(k)=n_k$ for each $k\in\omega$). Then $(B_n)_{n\in\omega\cup\{-1\}}$ is a partition of $\omega$ and $e_B^{-1}[B_n]=\{k\in\omega:\ n_k\in B_n\}\in\I$ for all $n\in\omega$ as $(\eps_k)$ is $\I$-converging to $0$. Therefore, there is $\alpha_0<\kappa(\I)$ such that: 
$$e_B^{-1}\left[\bigcup_{n\in\omega}A^{\alpha_0}_n\cap B_n\right]\notin\I.$$

Denote the above set by $C_{\alpha_0}$. We will show that:
$$C_{\alpha_0}\subseteq\left\{k\in\omega:\ |f_{n_k}(\phi(\alpha_0))|\geq\eps_k\right\},$$
which will contradict $f_{n_k}\xrightarrow{\I\text{QN}}0$. Fix $k\in C_{\alpha_0}$. Then there is $i\in\omega$ such that $n_k\in B_i$ and $n_k\in A^{\alpha_0}_i$. Hence, 
$$|f_{n_k}(\phi(\alpha_0))|\geq\frac{1}{i+1}>\eps_k.$$
\end{proof}

\begin{lem}
\label{dd-lem}
$\kappa(\I)\leq\dd$ for every weak P-ideal $\I$ on $\omega$.
\end{lem}

\begin{proof}
Let $\mathcal{F}=\{f_\alpha\in\omega^\omega:\ \alpha<\dd\}$ be a dominating family. Define finite sets: 
$$A^\alpha_n=\left\{k\in\omega:\ k\leq f_\alpha(n)\right\}$$
for each $\alpha<\dd$ and $n\in\omega$. We claim that the family $\{(A^\alpha_n):\ \alpha<\dd\}$ witnesses $\kappa(\I)\leq\dd$. 

Fix any partition $(B_n)_{n\in\omega\cup\{-1\}}$ of $\omega$ such that $e_{B}^{-1}[B_n]\in\I$ for each $n\in\omega$, where $B$ denotes the set $\bigcup_{n\in\omega}B_n$. Since $\I$ is a weak P-ideal, there is $C\notin\I$ with $C\cap e_{B}^{-1}[B_n]$ finite for all $n\in\omega$. Define a function $g\in\omega^\omega$ by: 
$$g(n)=\max\left(\left(e_B[C]\cap B_n\right)\cup\{0\}\right).$$

Since $\mathcal{F}$ is dominating, there is $\alpha_0<\dd$ with $g\leq^* f_{\alpha_0}$, i.e., $F=\{n\in\omega:\ g(n)>f_{\alpha_0}(n)\}$ is finite. As $(e_{B}^{-1}[B_n])$ is a partition of $\omega$, we have $\I\not\ni C=\bigcup_{n\in\omega}e_{B}^{-1}[B_n]\cap C$. Moreover, $\bigcup_{n\in\omega\setminus F}e_{B}^{-1}[B_n]\cap C\notin\I$ since $\bigcup_{n\in F}e_{B}^{-1}[B_n]\cap C$ is finite. We will show that:
$$\bigcup_{n\in\omega\setminus F}e_{B}^{-1}[B_n]\cap C\subseteq\bigcup_{n\in\omega}e_{B}^{-1}\left[B_n\cap A^{\alpha_0}_n\right].$$ 
Indeed, let $i\in e_{B}^{-1}[B_n]\cap C$ for some $n\in\omega\setminus F$. Then $e_B(i)\leq g(n)\leq f_{\alpha_0}(n)$. Hence, $e_B(i)\in A^{\alpha_0}_n$ and $i\in e_{B}^{-1}\left[B_n\cap A^{\alpha_0}_n\right]$. 
\end{proof}

\begin{lem}
\label{K-KB}
The following are equivalent for any ideal $\I$:
\begin{itemize}
	\item[(a)] $\I$ is a subset of some $\mathtt{F_\sigma}$ ideal;
	\item[(b)] $\I$ is $\leq_{KB}$-below some $\mathtt{F_\sigma}$ ideal;
	\item[(c)] $\I$ is $\leq_{K}$-below some $\mathtt{F_\sigma}$ ideal.	
\end{itemize}
\end{lem}

\begin{proof}
The implications (a)$\Rightarrow$(b) and (b)$\Rightarrow$(c) are obvious. To prove (c)$\Rightarrow$(a) suppose that $\I$ is an ideal on a set $M$ which is $\leq_{K}$-below $\mathtt{F_\sigma}$ ideal $\J$ on a set $N$. Let $f\colon N\to M$ be the witnessing function. Then $\bar{f}\colon\mathcal{P}(M)\to\mathcal{P}(N)$ given by $\bar{f}(A)=f^{-1}[A]$ for each $A\in\mathcal{P}(M)$ is continuous. Observe that $\bar{f}^{-1}[\J]$ is an ideal on $M$. Hence, it is an $\mathtt{F_\sigma}$ ideal, since $\J$ is $\mathtt{F_\sigma}$. Finally, $\I\subseteq\bar{f}^{-1}[\J]$, since $f$ witnesses $\I\leq_K\J$.
\end{proof}

\begin{lem}
\label{bb-lem}
$\kappa(\I)\leq\bb$ for every ideal $\I$ on $\omega$ which is $\leq_{K}$-below some $\mathtt{F_\sigma}$ ideal.
\end{lem}

\begin{proof}
Observe that $\I\subseteq\J$ implies $\kappa(\I)\leq\kappa(\J)$. Therefore, by Lemma \ref{K-KB}, we only need to show that $\kappa(\J)\leq\bb$ for every $\mathtt{F_\sigma}$ ideal $\J$. 

Let $\mathcal{F}=\{f_\alpha\in\omega^\omega:\ \alpha<\bb\}$ be an unbounded family. Without loss of generality we can assume that each $f_\alpha$ is non-decreasing (we may replace $f_\alpha$ with $\hat{f}_\alpha$ given by $\hat{f}_\alpha(n)=\max_{i\leq n}f(i)$ and observe that the family $\{\hat{f}_\alpha\in\omega^\omega:\ \alpha<\bb\}$ is unbounded as $f_\alpha\leq\hat{f}_\alpha$ for each $\alpha$). Define finite sets: 
$$A^\alpha_n=\left\{k\in\omega:\ k\leq f_\alpha(n)\right\}$$
for each $\alpha<\bb$ and $n\in\omega$. We claim that the family $\{(A^\alpha_n):\ \alpha<\bb\}$ witnesses $\kappa(\J)\leq\bb$. 

Let $\phi$ be the lower semi-continuous submeasure such that $\J=\Fin(\phi)$. Fix any partition $(B_n)_{n\in\omega\cup\{-1\}}$ of $\omega$ such that $e_{B}^{-1}[B_n]\in\J$ for each $n\in\omega$, where $B$ denotes the set $\bigcup_{n\in\omega}B_n$. 

Define a function $g\in\omega^\omega$ by $g(0)=0$ and
$$g(n)=\min\left\{k\in\omega:\ \phi\left(e_B^{-1}\left[[0,k]\right]\right)-\phi\left(e_B^{-1}\left[\bigcup_{i<n}B_i\right]\right)\geq n\right\}$$
for $n>0$. Note that $g$ is well-defined. Indeed, it follows from the facts that $\phi\left(e_B^{-1}\left[\bigcup_{i<n}B_i\right]\right)$ is finite and $\phi\left(e_B^{-1}\left[[0,k]\right]\right)$ tends to infinity as $k\rightarrow\infty$.

Recall that $\mathcal{F}$ is unbounded. Hence, there is $\alpha_0<\bb$ with $g(n)\leq f_{\alpha_0}(n)$ for infinitely many $n$. For each such $n$ we have: 
$$\phi\left(e_{B}^{-1}\left[\bigcup_{i\geq n}B_i\cap A^{\alpha_0}_n\right]\right)\geq
\phi\left(e_{B}^{-1}\left[A^{\alpha_0}_n\right]\right)-\phi\left(e_{B}^{-1}\left[\bigcup_{i<n}B_i\right]\right)\geq $$
$$\geq\phi\left(e_{B}^{-1}\left[[0,g(n)]\right]\right)-\phi\left(e_{B}^{-1}\left[\bigcup_{i<n}B_i\right]\right)\geq n.$$
Now it suffices to observe that:
$$\bigcup_{n\in\omega}e_{B}^{-1}\left[B_n\cap A^{\alpha_0}_n\right]\supseteq \bigcup_{n\in\omega}e_{B}^{-1}\left[\bigcup_{i\geq n}B_i\cap A^{\alpha_0}_n\right]\notin\J.$$
Indeed, if $k\in\bigcup_{n\in\omega}e_{B}^{-1}\left[\bigcup_{i\geq n}B_i\cap A^{\alpha_0}_n\right]$, then there are $n\in\omega$ and $i\geq n$ with $e_B(k)\in B_i\cap A^{\alpha_0}_n$. However, as $f_{\alpha_0}$ is non-decreasing, $A^{\alpha_0}_n\subseteq A^{\alpha_0}_i$. Therefore, $e_B(k)\in B_i\cap A^{\alpha_0}_i$.
\end{proof}

\begin{thm}
\label{main}
We have:
\begin{itemize}
	\item[(a)] $\bb\leq\non(\I\text{QN-space})\leq\non(\I\text{wQN-space})=\kappa(\I)\leq\dd$ for every weak P-ideal $\I$ on $\omega$;
	\item[(b)] $\non(\I\text{QN-space})=\non(\I\text{wQN-space})=\kappa(\I)=\bb$ for every ideal $\I$ on $\omega$ which is $\leq_{K}$-below some $\mathtt{F_\sigma}$ ideal.
\end{itemize}
\end{thm}

\begin{proof}
The equality $\non(\I\text{wQN-space})=\kappa(\I)$ follows immediately from Theorem \ref{nonIwQN} (when showing that $\non(\I\text{wQN-space})\leq\kappa(\I)$, it suffices to endow $X$ with the discrete topology). 

{\bf (a): }The first inequality follows from $\bb=\non(\text{QN-space})$ and the fact that every QN-space is an $\I$QN-space for any ideal $\I$. The second inequality is obvious, as every $\I$QN-space is an $\I$wQN-space. The third one is shown above and the last one is Lemma \ref{dd-lem}.

{\bf (b): }We have 
$$\bb=\non(\text{QN-space})\leq\non(\I\text{QN-space})\leq\non(\I\text{wQN-space})=\kappa(\I).$$
Moreover, $\kappa(\I)\leq\bb$ by Lemma \ref{bb-lem}.
\end{proof}

\begin{cor}
\label{P-ideals}
$\non(\I\text{QN-space})=\non(\I\text{wQN-space})=\bb$ for every analytic P-ideal $\I$ on $\omega$ generated by an unbounded submeasure.
\end{cor}

\begin{proof}
Any analytic P-ideal is of the form $\Exh(\phi)$ for some lower semi-continuous submeasure $\phi$. Therefore, it is contained in (so, in particular, $\leq_{K}$-below) $\Fin(\phi)$, which is $\mathtt{F_\sigma}$ (cf. \cite{Farah}). If $\phi$ is unbounded, then $\omega\notin\Fin(\phi)$, so $\Fin(\phi)$ becomes an ideal and we are done.
\end{proof}

\begin{rem}
Note that $\non(\I\text{QN-space})=\non(\I\text{wQN-space})=\bb$ also for some non-analytic ideals. Indeed, let $\I$ be a non-analytic ideal and consider the ideal $\Fin\oplus\I$ on $\{0,1\}\times\omega$ given by:
$$A\in\Fin\oplus\I\ \Longleftrightarrow\ \{n\in\omega:\ (0,n)\in A\}\in\Fin\ \wedge\ \{n\in\omega:\ (1,n)\in A\}\in\I$$
for every $A\subseteq\{0,1\}\times\omega$. Then $\Fin\oplus\I$ is not analytic and non-tall (so, by Theorem \ref{BDS}, a topological space is a QN-space if and only if it is an $(\Fin\oplus\I)$QN-space).
\end{rem}

\begin{cor}
\label{Marcin}
We have the following:
\begin{itemize}
	\item[(a)] $\bb(\I,\I,\Fin)=\bb$ for every ideal $\I$ on $\omega$ which is $\leq_{K}$-below some $\mathtt{F_\sigma}$ ideal. In particular, $\bb(\I,\I,\Fin)=\bb$ for all $\mathtt{F_\sigma}$ ideals and all analytic P-ideals generated by unbounded submeasures.
	\item[(b)] $\bb\leq\bb(\I,\I,\Fin)\leq\dd$ for every weak P-ideal $\I$ on $\omega$.
\end{itemize}
\end{cor}

\begin{proof}
By \cite[Section 6]{Supina} and the definition of $\bb(\I,\I,\Fin)$ (see \cite{Staniszewski}), we have $\bb(\I,\I,\Fin)=\non(\I\text{QN-space})$.
\end{proof}

\begin{thm}
\label{ex-IwQN}
If $\bb<\bb_\J$ for some ideal $\J$ on $\omega$, then there are a weak P-ideal $\I$ on $\omega$ and an $\I$wQN-space which is not a wQN-space. 
\end{thm}

Before proving the above, let us make a short comment.

\begin{rem}
Note that it is consistent with ZFC that $\bb<\bb_\J$ for some ideal $\J$ on $\omega$: in \cite{Canjar} it is proved (in ZFC) that there is a maximal ideal $\J$ with $\bb_\J$ equal to $\text{cf}(\dd)$ (the cofinality of $\dd$) and consistency of $\bb<\text{cf}(\dd)$ follows for instance from \cite[Theorem 2.5]{Blass}. Consistency of $\bb<\bb_\J$ may also be obtained under other set-theoretic assumptions -- see \cite{Mildenberger} for details.
\end{rem}

Theorem \ref{ex-IwQN} follows from the next Lemma as $\non(\text{wQN})=\bb$ by \cite[Corollary 3.2]{BRR}.

\begin{lem}
Let $\J$ be an ideal on $\omega$. Then there is a weak P-ideal $\I$ on $\omega$ such that $\non(\I\text{wQN-space})\geq\bb_\J$. 
\end{lem}

\begin{proof}
Define an ideal $\bar{\I}$ on $\omega\times\omega$ by:
$$A\in\bar{\I}\ \Longleftrightarrow\ \{n\in\omega:\ |A\cap(\{n\}\times\omega)|=\omega\}\in\Fin\ \wedge\ \{n\in\omega:\ A\cap(\{n\}\times\omega)\neq\emptyset\}\in\J$$
for each $A\subseteq\omega\times\omega$ (i.e., $\bar{\I}=(\Fin\otimes\Fin)\cap(\J\otimes\emptyset)$). Note that $\bar{\I}$ is an ideal as an intersection of two ideals.

We need to show two facts:
\begin{itemize}
	\item[(i)] $\bar{\I}$ is a weak P-ideal;
	\item[(ii)] $\non(\I\text{wQN-space})\geq\bb_\J$ for every ideal $\I$ on $\omega$ isomorphic to $\bar{\I}$.
\end{itemize}
Then any ideal on $\omega$ isomorphic to $\bar{\I}$ will be as needed (since being a weak P-ideal is invariant over isomorphisms of ideals).

{\bf (i) $\bar{\I}$ is a weak P-ideal: }Fix a partition $(X_n)$ of $\omega\times\omega$ into sets belonging to $\bar{\I}$. Define by induction two sequences $(m_n),(k_n)\subseteq\omega$ such that for each $n\in\omega$ we have $(n,m_n)\in X_{k_n}$ and
$$m_n=\left\{\begin{array}{ll}
\min\{m\in\omega:(n,m)\notin\bigcup\{X_{k_i}:i<n\}\} & \mbox{\boldmath{if }}\{n\}\times\omega\not\subseteq\bigcup_{i<n}X_{k_i},\\
\min\{m\in\omega:(n,m)\in\bigcup\{X_{k}:|X_k\cap(\{n\}\times\omega)|=\omega\}\} & \mbox{\boldmath{otherwise.}}\\
\end{array}\right.$$
Then $Y=\{(n,m_n):\ n\in\omega\}\notin\I$ as $\{n\in\omega:\ Y\cap(\{n\}\times\omega)\neq\emptyset\}=\omega\notin\J$. Moreover, $Y\cap X_n$ is finite for all $n$ (otherwise we would have $|X_n\cap(\{k\}\times\omega)|=\omega$ for infinitely many $k\in\omega$).

{\bf (ii) $\non(\I\text{wQN-space})\geq\bb_\J$: }Fix any bijection $\phi\colon\omega\to\omega\times\omega$ and denote $X_n=\phi^{-1}[\{n\}\times\omega]$ for all $n\in\omega$. We will show that $\bb_\J\leq\non(\I_\phi\text{wQN-space})$ where $\I_\phi=\{\phi^{-1}[A]:\ A\in\bar{\I}\}$ is an ideal on $\omega$ isomorphic to $\bar{\I}$. 

We will use the equality $\non(\I\text{wQN-space})=\kappa(\I)$ from Theorem \ref{main}. Let $\kappa<\bb_\J$ and $\{(A^\alpha_n):\ \alpha<\kappa\}\subseteq\Fin^\omega$. Define:
$$f_\alpha(n)=\max\left\{k\in\omega:\ (n,k)\in\phi[A^\alpha_n]\right\}$$
for all $\alpha<\kappa$ and $n\in\omega$. Then there is $g\in\omega^\omega$ such that $\{n\in\omega:\ f_\alpha(n)>g(n)\}\in\J$ for each $\alpha<\kappa$.  

Now we proceed to the construction of a partition $(B_n)_{n\in\omega\cup\{-1\}}$. Define:
$$C_{-1}=\{(i,j):\ j\leq g(i)\};$$
$$C_n=\{(n,j):\ j>g(n)\}$$
for each $n\in\omega$. Observe that $\phi^{-1}[C_n]\notin\Fin$ and $\phi^{-1}[C_n]\subseteq X_n$ for each $n$. Define a sequence $(m_k)\in\omega^\omega$ by $k\in X_{m_k}$. Pick inductively a sequence $(n_k)\subseteq\omega$ such that:
$$n_0=\min\left\{n\in\omega:\ n\in X_{m_0}\cap \phi^{-1}[C_{m_0}]\right\};$$
$$n_k=\min\left\{n>n_{k-1}:\ n\in X_{m_k}\cap \phi^{-1}[C_{m_k}]\right\}.$$
Denote:
$$B=\{n_k:\ k\in\omega\},\quad B_{-1}=\omega\setminus B\quad\text{ and }\quad B_n=B\cap X_n\text{ for each }n\in\omega.$$

Notice that $e_B(k)=n_k$ and for each $n\in\omega$ we have:
\begin{itemize}
	\item[(a)] $B_n\subseteq X_n$;
	\item[(b)] $B_n\subseteq \phi^{-1}[C_n]$;
	\item[(c)] $n_k\in B_{m_k}$.
\end{itemize}
Moreover, by item (c), we have: 
$$k\in e_B^{-1}[B_n]\ \Longleftrightarrow\ n_k\in B_n\ \Longleftrightarrow\ n=m_k\ \Longleftrightarrow\ k\in X_n,$$
which establishes:
\begin{itemize}
	\item[(d)] $e_B^{-1}[B_n]=X_n\in\I_\phi$.
\end{itemize}

By item (d), the partition $(B_n)_{n\in\omega\cup\{-1\}}$ will be as needed, provided that we will show $\bigcup_{k\in\omega}e_B^{-1}[A^\alpha_k\cap B_k]\in\I_\phi$ for each $\alpha<\kappa$. 

Fix any $\alpha<\kappa$. Then $Y_\alpha=\{n\in\omega:\ f_\alpha(n)>g(n)\}\in\J$. Let $n\in\omega$.

If $n\notin Y_\alpha$, we have:
$$\phi[A^\alpha_n]\cap C_n=\emptyset\ \Longrightarrow\ A^\alpha_n\cap\phi^{-1}[C_n]=\emptyset\ \Longrightarrow\ A^\alpha_n\cap B_n=\emptyset\ \Longrightarrow\ e_B^{-1}[A^\alpha_n\cap B_n]=\emptyset.$$
Indeed, the second implication follows from condition (b) and the remaining two are trivial. By item (d) and the fact that $(X_n)$ is a partition of $\omega$, we get that: 
$$\left(\bigcup_{k\in\omega}e_B^{-1}[A^\alpha_k\cap B_k]\right)\cap X_n=\left(\bigcup_{k\in\omega}e_B^{-1}[A^\alpha_k]\cap X_k\right)\cap X_n=$$
$$=e_B^{-1}[A^\alpha_n]\cap X_n=e_B^{-1}[A^\alpha_n\cap B_n]=\emptyset.$$

On the other hand, if $n\in Y_\alpha$, then we have $e_B^{-1}[A^\alpha_n\cap B_n]\in\Fin$ as $A^\alpha_n\in\Fin$. Therefore, again by item (d) and the fact that $(X_n)$ is a partition of $\omega$ we get that $\left(\bigcup_{k\in\omega}e_B^{-1}[A^\alpha_k\cap B_k]\right)\cap X_n\in\Fin$.

Hence, $\bigcup_{k\in\omega}e_B^{-1}[A^\alpha_k\cap B_k]\in\I_\phi$.
\end{proof}

Let $\I$ be an ideal on $\omega$. A sequence $(U_n)$ of subsets of a topological space $X$ is an \emph{$\I\text{-}\gamma$-cover} if $U_n\neq X$ for all $n\in\omega$ and $\{n\in\omega:\ x\notin U_n\}\in\I$ for all $x\in X$. By $\I\text{-}\Gamma$ we denote the family of all open $\I\text{-}\gamma$-covers. We write $\Gamma$ instead of $\Fin\text{-}\Gamma$. Moreover, $X$ is $S_1(\Gamma, \I\text{-}\Gamma)$ whenever for every sequence $(\mathcal{U}_n)\subseteq\Gamma$ one can find $U_n\in\mathcal{U}_n$, for $n\in\omega$, with $(U_n)\in\I\text{-}\Gamma$.

The \emph{Scheepers Conjecture} asserts that a space is a wQN-space if and only if it satisfies $S_1(\Gamma,\Gamma)$ (cf. \cite{Scheepers}). It is still open whether the Scheepers Conjecture is provable, however Dow showed that it is consistently true (cf. \cite{Dow}).

\v{S}upina proved in \cite[Corollary 1.7]{Supina} that the ideal version of Scheepers Conjecture does not hold if $\I$ is not a weak P-ideal as in this case one can find a perfectly normal $\I$wQN-space which is not $S_1(\Gamma, \I\text{-}\Gamma)$. However, by Theorem \ref{Supina2}, if $\I$ is not a weak P-ideal, then any topological space is an $\I$wQN-space, so the above result is not rewarding.

The following result shows that the ideal version of Scheepers Conjecture for weak P-ideals consistently does not hold.

\begin{cor}
If $\bb<\bb_\J$ for some ideal $\J$ on $\omega$, then there are a weak P-ideal $\I$ on $\omega$ and an $\I$wQN-space which is not $S_1(\Gamma, \I\text{-}\Gamma)$. 
\end{cor}

\begin{proof}
By \cite[Corollary 7.4(ii)]{Supina}, ${\non}(S_1(\Gamma, \hat{\I}\text{-}\Gamma)\text{-space})=\bb_{\hat{\I}}$ for every ideal $\hat{\I}$. Let $\I$ be the ideal from Theorem \ref{ex-IwQN}. We will show that $\bb_\I=\bb$. Observe that if $\I_1\subseteq\I_2$, then $\bb_{\I_1}\leq\bb_{\I_2}$. Moreover, $\I\subseteq\Fin\otimes\Fin$. Hence, it suffices to show that $\bb_{\Fin\otimes\Fin}=\bb$. This follows from the fact that $\Fin\otimes\Fin$ is a Borel (in fact $\mathtt{F_{\sigma\delta\sigma}}$) ideal. Indeed, by the proof of \cite[Corollary 5.5]{Farkas-Soukup}, we have $\bb_{\hat{\I}}=\bb$ for any ideal $\hat{\I}$ which is $\leq_{RB}$-above $\Fin$ and this is the case for every Borel ideal by \cite[Corollary 3.10.2]{Farah}.
\end{proof}

\section{Relation between $\I$wQN-spaces and wQN-spaces}

In \cite[Problem 3.2]{BDS} authors ask about existence of a tall ideal $\I$ such that for any sequence of functions $\I$QN-converging to $0$ one can find its subsequence converging quasi-normally to $0$. In this section we investigate this property.

Let $\I$ be a tall ideal on $\omega$. Define:
$$\cov(\I)=\min\left\{|\mathcal{A}|:\ \mathcal{A}\subseteq\I\ \wedge\ \forall_{\textit{S}\in[\omega]^\omega}\ \exists_{\textit{A}\in\mathcal{A}}\ |\textit{A}\cap \textit{S}|=\omega\right\}.$$
This cardinal invariant was considered e.g. in \cite{Brendle-Shelah} (where a different notation is used) and \cite{Hrusak2}. It is a variation of the pseudointersection number $\pp$ -- we additionally require that the witnessing family is from the filter dual to the ideal $\I$.

\begin{rem}
\label{p+c}
We have $\pp\leq\cov(\I)\leq\cc$ for every tall ideal $\I$ (cf. \cite{Hrusak}).
\end{rem}

There are examples of tall ideals $\I$ with non-trivial values of $\cov(\I)$, for instance:
\begin{itemize}
	\item $\cov(\Fin\otimes\Fin)=\bb$ (cf. \cite{Hrusak});
	\item $\cov({\tt nwd})={\tt cov}(\mathcal{M})$, where {\tt nwd} is the ideal on $\mathbb{Q}\cap [0,1]$ consisting of all nowhere dense subsets of $\mathbb{Q}\cap [0,1]$ (cf. \cite{Hrusak} or \cite{Keremedis});
	\item $\cov(\mathcal{ED})={\tt non}(\mathcal{M})$, where $\mathcal{ED}$ is the ideal on $\omega\times\omega$ generated by all vertical lines (i.e., sets $\{n\}\times\omega$ for $n\in\omega$) and graphs of functions from $\omega$ to $\omega$ (cf. \cite{Hrusak} or \cite{Hrusak3});
	\item $\cov({\tt conv})=\cc$, where ${\tt conv}$ is the ideal on $\mathbb{Q}\cap [0,1]$ generated by sequences in $\mathbb{Q}\cap [0,1]$ convergent in $[0,1]$ (cf. \cite{Hrusak} or \cite{Meza}).
\end{itemize}
For more examples see \cite{Hrusak}.

\begin{thm}
\label{thm-cov}
Let $\I$ be a tall ideal on $\omega$. The following are equivalent for any set $X$:
\begin{itemize}
	\item[(a)] $|X|<\cov(\I)$;
	\item[(b)] for any sequence of real-valued functions defined on $X$, if it $\I$QN-converges to some $f\in\mathbb{R}^X$, then one can find its subsequence QN-converging to $f$.
\end{itemize}
\end{thm}

\begin{proof}
{\bf (a)$\Rightarrow$(b): }Suppose that $|X|<\cov(\I)$ and fix $(f_n)\subseteq\mathbb{R}^X$ which $\I$QN-converges to some $f\in\mathbb{R}^X$ with the witnessing sequence $(\eps_n)\subseteq (0,1)$. 

For each $x\in X$ let: 
$$B_x=\left\{n\in\omega:\ |f_n(x)-f(x)|\geq\eps_n\right\}\in\I.$$
Let also: 
$$A_k=\left\{n\in\omega:\ \frac{1}{k+2}\leq\eps_n<\frac{1}{k+1}\right\}\in\I$$
for all $k\in\omega$.

Observe that $|X|+\omega<\cov(\I)$ as $\cov(\I)\geq\pp>\omega$ (cf. Remark \ref{p+c}). Hence, there is an infinite set $S\subseteq\omega$ which has finite intersections with all $B_x$, $x\in X$, as well as with all $A_k$, $k\in\omega$. Let $(n_m)$ be an increasing enumeration of $S$. 

We will show that $(f_{n_m})$ QN-converges to $f$. Define $\eps'_m=\eps_{n_m}$ for all $m\in\omega$. Then $\eps'_m$ converges to $0$ as $S\cap A_k$ is finite for each $k\in\omega$. Moreover, we have
$$|\left\{m\in\omega:\ |f_{n_m}(x)-f(x)|\geq\eps'_m\right\}|=|S\cap B_x|<\omega$$
for all $x\in X$.

{\bf (b)$\Rightarrow$(a): }Suppose that $|X|\geq\cov(\I)$. Let $\phi\colon\cov(\I)\to X$ be an injection. Let also $\{A_\alpha:\ \alpha<\cov(\I)\}$ be such a family of members of $\I$ that for each infinite $S\subseteq\omega$ there is $\alpha<\cov(\I)$ with $S\cap A_\alpha$ infinite. 

Define a sequence $(f_n)\subseteq\mathbb{R}^X$ by: 
$$f_n\left(x\right)=\left\{\begin{array}{ll}
1 & \mbox{\boldmath{if }} n\in A_\alpha\mbox{\boldmath{ and }}x=\phi(\alpha)\mbox{\boldmath{ for some }}\alpha<\cov(\I),\\
0 & \mbox{\boldmath{otherwise.}}\\
\end{array}\right.$$
Let also $f\in\mathbb{R}^X$ be the function constantly equal to $0$.

It is easy to see that for any $(\eps_n)\subseteq(0,1)$ converging to $0$ we have: 
$$\left\{n\in\omega:\ |f_{n}(x)-f(x)|\geq\eps_n\right\}=\emptyset$$
for each $x\in X\setminus\phi[\cov(\I)]$ and

$$\left\{n\in\omega:\ |f_{n}(x)-f(x)|\geq\eps_n\right\}=A_\alpha\in\I$$
for each $x=\phi(\alpha)\in\phi[\cov(\I)]$. Hence, $(f_n)$ is $\I$QN-convergent to $f$. 

Now we will show that none of subsequences of $(f_n)$ converges to $f$. Fix any subsequence $(f_{n_m})\subseteq(f_n)$. The set $S=\left\{n_m:\ m\in\omega\right\}$ is infinite, so there is $\alpha_0<\cov(\I)$ with $S\cap A_{\alpha_0}$ infinite. Then we have: 
$$|\left\{m\in\omega:\ |f_{n_m}(\phi(\alpha_0))-f(\phi(\alpha_0))|\geq 1\right\}|=|S\cap A_{\alpha_0}|=\omega.$$
Therefore, $(f_{n_m})$ cannot QN-converge to $f$.
\end{proof}

As a consequence of the above Theorem we obtain the main result of this Section.

\begin{cor}
Let $\I$ be a tall ideal on $\omega$ and $X$ be a topological space of cardinality less than $\cov(\I)$. Then $X$ is an $\I$wQN-space if and only if it is a wQN-space.
\end{cor}

\begin{proof}
Any wQN-space is an $\I$wQN-space and the other implication is an immediate consequence of the previous Theorem.
\end{proof}

\begin{cor}
We have the following:
\begin{itemize}
	\item[(a)] If $|X|<\pp$, then for any ideal $\I$ on $\omega$ and any sequence of real-valued functions defined on $X$, which $\I$QN-converges to some $f\in\mathbb{R}^X$, one can find its subsequence QN-converging to $f$.
	\item[(b)] If $|X|\geq\cc$, then for any tall ideal $\I$ on $\omega$ there is a sequence of real-valued functions defined on $X$, which $\I$QN-converges to some $f\in\mathbb{R}^X$, but none of its subsequences QN-converges to $f$.
\end{itemize}
\end{cor}

\begin{proof}
First item in the case of non-tall ideals follows from Theorem \ref{BDS}. To prove the remaining parts it suffices to observe that $\pp\leq\cov(\I)\leq\cc$ for any tall ideal $\I$ on $\omega$ (cf. Remark \ref{p+c}).
\end{proof}

The anonymous referee of this paper had pointed out that item (b) of the above result can be strengthened in the following way.

\begin{prop}
\label{recenzent-prop1}
If $|X|\geq\cc$, then there is a sequence $(f_n)$ of real-valued functions defined on $X$ such that for every tall ideal $\I$ on $\omega$:
\begin{itemize}
 \item the set $A(\I)=\{x\in X:\ f_n(x)\xrightarrow{\I}0\}$ has cardinality $\cc$;
 \item $f_n\upharpoonright_{A(\I)}\xrightarrow{\I\text{QN}} 0$;
 \item there is $B(\I)\subseteq A(\I)$ of cardinality $\cov(\I)$ such that $(f_n)$ has no subsequence QN-converging to $0$ on $B(\I)$.
\end{itemize}
\end{prop}

\begin{proof}
Fix a surjection $\phi\colon X\to[\omega]^\omega$ and define a sequence $(f_n)\subseteq\mathbb{R}^X$ by: 
$$f_n\left(x\right)=\left\{\begin{array}{ll}
1 & \mbox{\boldmath{if }} n\in h(x),\\
0 & \mbox{\boldmath{otherwise.}}\\
\end{array}\right.$$
Then given any $x\in X$ we have $\{n\in\omega:\ |f_n(x)|\geq\eps\}=h(x)$ for each $\eps\in(0,1)$. Therefore, $A(\I)=h^{-1}[\I]$ for every ideal $\I$ on $\omega$ and $f_n\upharpoonright_{A(\I)}\xrightarrow{\I\text{QN}} 0$. If $\I$ is tall, then it has cardinality $\cc$ (a tall ideal must have an infinite member and all infinite subsets of that member belong to the ideal as well). Hence, $|A(\I)|=\cc$. 

Let $\mathcal{A}\subseteq[\omega]^\omega$ be the family from the definition of $\cov(\I)$. Find $B(\I)\subseteq A(\I)$ of cardinality $\cov(\I)$ such that $h[B(\I)]=\mathcal{A}$. Now we will show that none of subsequences of $(f_n)$ converges to $0$ on $B(\I)$. Fix any subsequence $(f_{n_m})\subseteq(f_n)$. The set $S=\left\{n_m:\ m\in\omega\right\}$ is infinite, so one can find $x\in B(\I)$ with $S\cap h(x)$ infinite (since $h[B(\I)]=\mathcal{A}$). Then we have: 
$$|\left\{m\in\omega:\ |f_{n_m}(x)|\geq 1\right\}|=|S\cap h(x)|=\omega.$$
\end{proof}

\section{Additivity of $\I$QN-spaces}

Recall that for an ideal $\I$ on $\omega$ $\add(\I\text{QN-space})$ denotes the minimal cardinal $\kappa$ such that there is a perfectly normal non-$\I$QN-space which can be expressed as a union of $\kappa$ many $\I$QN-spaces.

\begin{df}
For an ideal $\I$ on $\omega$ denote by $\mathcal{P}_\I$ the family of all partitions of $\omega$ into sets belonging to $\I$. We define: 
$$\lambda(\I)=\min\left\{|\mathcal{A}|:\ \mathcal{A}\subseteq\mathcal{P}_\I\ \wedge\ \forall_{(B_n)\in\mathcal{P}_\I}\ \exists_{(A_n)\in\mathcal{A}}\ \bigcup_{n\in\omega}\left(B_n\cap\bigcup_{k<n}A_k\right)\notin\I\right\}.$$
\end{df}

\begin{rem}
\label{add-remark}
Observe that $\lambda(\I)\leq\cc$ for every ideal $\I$ on $\omega$, i.e., $\lambda(\I)$ is well defined. Indeed, take $\mathcal{A}=\mathcal{P}_\I$. Then $|\mathcal{A}|\leq\cc$ and for each $(B_n)\in\mathcal{P}_\I$ the partition $(A_n)\in\mathcal{A}$ given by $A_0=B_0\cup B_1$ and $A_{n}=B_{n+1}$ for $n\in\omega\setminus\{0\}$ is such that: 
$$\bigcup_{n\in\omega}\left(B_n\cap\bigcup_{k<n}A_k\right)=\omega\setminus B_0\notin\I.$$
\end{rem}

\begin{rem}
\label{add-remark2}
Observe that $\lambda(\I)\geq\omega_1$ for every ideal $\I$ on $\omega$. Indeed, fix any $\{(A_n^m)\in\mathcal{P}_\I:\ m<\omega_1\}$ and define 
$$B_n=\left(\bigcup_{m\leq n}\bigcup_{k\leq n}A^m_k\right)\setminus\left(\bigcup_{m<n}\bigcup_{k<n}A^m_k\right)$$
for each $n\in\omega$. Then each $B_n$ belongs to $\I$ (as a subset of a finite union of sets belonging to $\I$). Moreover, $(B_n)$ is a partition of $\omega$ (as $(A^0_n)$ is a partition of $\omega$ and $A^0_n\subseteq B_0\cup\ldots\cup B_n$ for each $n\in\omega$). Now it suffices to observe that for each $m\in\omega$ we have:
$$\bigcup_{n\in\omega}\left(B_n\cap\bigcup_{k<n}A^m_k\right)=\bigcup_{n\leq m}\left(B_n\cap\bigcup_{k<n}A^m_k\right)\in\I.$$
\end{rem}

\begin{thm}
\label{add-prop1}
The following are equivalent for any ideal $\I$ on $\omega$:
\begin{itemize}
	\item[(a)] $\kappa<\lambda(\I)$;
	\item[(b)] if $X=\bigcup_{\alpha<\kappa}X_\alpha$ and $(f_n)\subseteq\mathbb{R}^X$ $\I$QN-converges to $f\in\mathbb{R}^X$ on each $X_\alpha$, then $(f_n)$ $\I$QN-converges to $f$ on $X$.
\end{itemize}
\end{thm}

\begin{proof}
{\bf (a)$\Rightarrow$(b): }Fix $X$ and $X_\alpha\subseteq X$, for $\alpha<\kappa<\lambda(\I)$, with $X=\bigcup_{\alpha<\kappa}X_\alpha$. Let $(f_n)\subseteq\mathbb{R}^X$ and $f\in\mathbb{R}^X$. Suppose that $(f_n)$ $\I$QN-converges to $f$ on each $X_\alpha$ with the witnessing sequence $(\eps_n^\alpha)\subseteq(0,1)$. Define: 
$$A^\alpha_k=\left\{n\in\omega:\ \frac{1}{k+2}\leq\eps^\alpha_n<\frac{1}{k+1}\right\}$$
for each $n\in\omega$ and $\alpha<\kappa$. Then $(A^\alpha_n)_{n\in\omega}$, for any $\alpha$, is a partition of $\omega$ and each $A^\alpha_k$ belongs to $\I$ since $(\eps_n^\alpha)$ is $\I$-convergent to $0$. As $\kappa<\lambda(\I)$, there is $(B_n)\in\mathcal{P}_\I$ such that $\bigcup_{n\in\omega}\left(B_n\cap\bigcup_{k<n}A^\alpha_k\right)\in\I$ for each $\alpha$.

Define a sequence $(\eps_k)\subseteq(0,1)$ by: 
$$\eps_k=\frac{1}{n+1}\ \Longleftrightarrow\ k\in B_n.$$
Then $(\eps_k)$ is $\I$-convergent to $0$. We will show that it witnesses $f_n\xrightarrow{\I\text{QN}}f$ on $X$.

Fix $x\in X$ and let $\alpha_0<\kappa$ be such that $x\in X_{\alpha_0}$. We have:
$$\left\{k\in\omega:\ |f_k(x)-f(x)|\geq\eps_k\right\}\subseteq$$ $$\subseteq\left\{k\in\omega:\ \eps_k<\eps_k^{\alpha_0}\right\}\cup\left\{k\in\omega:\ |f_k(x)-f(x)|\geq\eps^{\alpha_0}_k\right\}.$$
The latter set belongs to $\I$ since $(\eps_n^{\alpha_0})$ witnesses $f_n\xrightarrow{\I\text{QN}}f$ on $X_{\alpha_0}$. Now it suffices to show that $\left\{k\in\omega:\ \eps_k<\eps_k^{\alpha_0}\right\}\in\I$. Indeed, we have:
$$\left\{k\in\omega:\ \eps_k<\eps_k^{\alpha_0}\right\}=\bigcup_{n\in\omega}\left(B_n\cap\bigcup_{k<n}A^{\alpha_0}_k\right)\in\I.$$

{\bf (b)$\Rightarrow$(a): }Fix any set $X$ of cardinality at least $\lambda(\I)$, let $\phi\colon\lambda(\I)\to X$ be an injection and define $X_\alpha=\{\phi(\alpha)\}$ for each $\alpha<\lambda(\I)$. Set also a family $\{(A^\alpha_n)_{n\in\omega}:\ \alpha<\lambda(\I)\}\subseteq\mathcal{P}_\I$ such that for every $(B_n)\in\mathcal{P}_\I$ one can find $\alpha<\lambda(\I)$ with $\bigcup_{n\in\omega}\left(B_n\cap\bigcup_{k<n}A^\alpha_k\right)\notin\I$. Define a sequence of functions $(f_n)\in\mathbb{R}^X$ by:
$$f_k\left(x\right)=\left\{\begin{array}{ll}
\frac{1}{n+1} & \mbox{\boldmath{if }} k\in A^\alpha_n\mbox{\boldmath{ and }}x=\phi(\alpha)\mbox{\boldmath{ for some }}\alpha<\lambda(\I),\\
0 & \mbox{\boldmath{otherwise.}}\\
\end{array}\right.$$
Then $f_n\xrightarrow{\I\text{QN}}0$ on each $X_\alpha$ as $(A^\alpha_n)_{n\in\omega}\in\mathcal{P}_\I$. 

We will show that $(f_n)$ does not $\I$QN-converge to $0$ on $X$. Suppose to the contrary that $f_n\xrightarrow{\I\text{QN}}0$ on $X$ and it is witnessed by some sequence $(\eps_k)\subseteq(0,1)$. Define:
$$B_n=\left\{k\in\omega:\ \frac{1}{n+2}\leq\eps_k<\frac{1}{n+1}\right\}$$
for all $n\in\omega$. Then $(B_n)\in\mathcal{P}_\I$. Hence, by the definition of $\{(A^\alpha_n)_{n\in\omega}:\ \alpha<\lambda(\I)\}$, there is $\alpha_0<\lambda(\I)$ such that $C=\bigcup_{n\in\omega}\left(B_n\cap\bigcup_{k<n}A^{\alpha_0}_k\right)\notin\I$. We will show that $C\subseteq\{k\in\omega:\ |f_k(\phi(\alpha_0))|\geq\eps_k\}$, which will end the proof. Fix any $m\in C$. Then there is $n\in\omega$ such that $m\in B_n\cap\bigcup_{k<n}A^{\alpha_0}_k$. Therefore, $\eps_m<\frac{1}{n+1}$ and $f_m(\phi(\alpha_0))\geq\frac{1}{n+1}$. Hence, $m\in\{k\in\omega:\ |f_k(\phi(\alpha_0))|\geq\eps_k\}$.
\end{proof}

From the above Proposition we can easily derive a connection between $\lambda(\I)$ and $\add(\I\text{QN-space})$.

\begin{cor}
\label{add-cor}
$\lambda(\I)\leq\add(\I\text{QN-space})$ for every ideal $\I$ on $\omega$.
\end{cor}

\begin{proof}
Obvious.
\end{proof}

\begin{rem}
Notice that $\lambda(\I)$ and $\add(\I\text{QN-space})$ are not the same. Indeed, by Theorem \ref{Supina2}, any topological space is a $(\Fin\otimes\Fin)$QN-space. Therefore, it does not make sense to consider additivity of $(\Fin\otimes\Fin)$QN-spaces. However, $\lambda(\Fin\otimes\Fin)\leq\cc$ by Remark \ref{add-remark}.
\end{rem}

Now we proceed to obtaining a lower and upper bounds for $\add(\I\text{QN-space})$ and $\lambda(\I)$. 

\begin{cor}
\label{add-cor2}
We have:
\begin{itemize}
	\item[(a)] $\omega_1\leq\lambda(\I)\leq\add(\I\text{QN-space})\leq\dd$ for every weak P-ideal $\I$ on $\omega$;
	\item[(b)] $\omega_1\leq\add(\I\text{QN-space})\leq\bb$ for every ideal $\I$ on $\omega$ which is $\leq_{K}$-below some $\mathtt{F_\sigma}$ ideal;
	\item[(c)] $\lambda(\I)=\add(\I\text{QN-space})=\bb$ for every P-ideal $\I$ on $\omega$ which is $\leq_{K}$-below some $\mathtt{F_\sigma}$ ideal. In particular, $\lambda(\I)=\add(\I\text{QN-space})=\bb$ for every analytic P-ideal generated by an unbounded submeasure.
\end{itemize}
\end{cor}

\begin{proof}
{\bf (a): }In fact, for every weak P-ideal $\I$ on $\omega$ we have the following sequence of inequalities:
$$\omega_1\leq\lambda(\I)\leq\add(\I\text{QN-space})\leq\non(\I\text{QN-space})\leq\dd.$$
Indeed, the first inequality is Remark \ref{add-remark2}, the second one is Corollary \ref{add-cor}, the third inequality is obvious (as all singleton spaces are $\I$QN-spaces for every ideal $\I$) and the last one follows from item (a) of Theorem \ref{main}.

{\bf (b): }It suffices to use item (b) of Theorem \ref{main} instead of (a) in the above considerations. 

{\bf (c): }The first part is a combination of item (b) and \cite[Theorem 2.2]{Dascha} stating that $\lambda(\I)\geq\bb$ for all P-ideals $\I$. Analytic P-ideals generated by unbounded submeasures are $\leq_{K}$-below some $\mathtt{F_\sigma}$ ideal (cf. the proof of Corollary \ref{P-ideals}) and $\mathtt{F_\sigma}$ P-ideals are generated by unbounded submeasures.
\end{proof}

By Theorem \ref{BDS} we have $\add(\I\text{QN-space})=\bb$ for all non-tall ideals $\I$. By the last item of Corollary \ref{add-cor2} we also have $\add(\I\text{QN-space})=\bb$ for analytic P-ideals $\I$ on $\omega$ which are $\leq_{K}$-below some $\mathtt{F_\sigma}$ ideal. We want to end this section with an example of a class of tall ideals $\I$ which are not P-ideals and satisfy $\add(\I\text{QN-space})=\bb$.

\begin{prop}
\label{lprop}
For any ideal $\I$ on $\omega$ we have $\lambda(\I\otimes\emptyset)=\lambda(\I)$.
\end{prop}

\begin{proof}
Define $f\colon\mathcal{P}_{\I}\to\mathcal{P}_{\I\otimes\emptyset}$ by $f((A_n))=(A^f_n)=(A_n\times\omega)$ and $g\colon\mathcal{P}_{\I\otimes\emptyset}\to\mathcal{P}_{\I}$ by $g((A_n))=(A^g_n)$, where 
$$i\in A^g_n\ \Longleftrightarrow\ n=\min\{k\in\omega:\ A_k\cap(\{i\}\times\omega)\neq\emptyset\}.$$

First we will show that $\lambda(\I\otimes\emptyset)\leq\lambda(\I)$. Let $\mathcal{A}\subseteq\mathcal{P}_\I$ be as in the definition of $\lambda(\I)$. We claim that $f[\mathcal{A}]\subseteq\mathcal{P}_{\I\otimes\emptyset}$ is as needed. Fix any $(B_n)\in\mathcal{P}_{\I\otimes\emptyset}$. Then there is $(A_n)\in\mathcal{A}$ with $\bigcup_{n\in\omega}\left(B_n^g\cap\bigcup_{k<n}A_k\right)\notin\I$. Note that for each $i\in B_n^g$ there is $j\in\omega$ with $(i,j)\in B_n$. If additionally $i\in\bigcup_{k<n}A_k$, then $(i,j)\in\bigcup_{k<n}A^f_k$. Hence, $\bigcup_{n\in\omega}\left(B_n\cap\bigcup_{k<n}A^f_k\right)\notin\I\otimes\emptyset$.

Now we will show that $\lambda(\I)\leq\lambda(\I\otimes\emptyset)$. Let $\mathcal{A}\subseteq\mathcal{P}_{\I\otimes\emptyset}$ be as in the definition of $\lambda(\I\otimes\emptyset)$. We claim that $g[\mathcal{A}]\subseteq\mathcal{P}_{\I}$ is as needed. Fix any $(B_n)\in\mathcal{P}_{\I}$. Then there is $(A_n)\in\mathcal{A}$ with $\bigcup_{n\in\omega}\left(B_n^f\cap\bigcup_{k<n}A_k\right)\notin\I\otimes\emptyset$. Hence, 
$$\bigcup_{n\in\omega}\left(B_n\cap\bigcup_{k<n}A^g_k\right)=\bigcup_{n\in\omega}\left\{i\in B_n:\ \exists_{j\in\omega}\ (i,j)\in\bigcup_{k<n}A_k\right\}\notin\I.$$
\end{proof}

\begin{lem}
\label{llem}
$\I\otimes\emptyset\leq_{KB}\I$ for any ideal $\I$ on $\omega$.
\end{lem}

\begin{proof}
We claim that the function $f\colon\omega\to\omega\times\omega$ given by $f(n)=(n,0)$, for $n\in\omega$, witnesses $\I\otimes\emptyset\leq_{KB}\I$. Indeed, $f$ is finite-to-one (even one-to-one) and given any $A\in\I\otimes\emptyset$ we have:
$$f^{-1}[A]\subseteq\{n\in\omega:\ A\cap(\{n\}\times\omega)\neq\emptyset\}\in\I.$$
\end{proof}

\begin{cor}
\label{add-cor3}
Let $\I$ be a P-ideal and $\J$ be any ideal on $\omega$ isomorphic to $\I\otimes\emptyset$. Then we have:
\begin{itemize}
	\item[(a)] if $\non(\I\text{QN-space})=\bb$, then $$\add(\J\text{QN-space})=\non(\J\text{QN-space})=\bb;$$
	\item[(b)] if $\I$ is $\leq_{K}$-below some $\mathtt{F_\sigma}$ ideal (in particular, if $\I$ is generated by an unbounded submeasure), then $$\add(\J\text{QN-space})=\non(\J\text{QN-space})=\non(\J\text{wQN-space})=\bb.$$
\end{itemize}
\end{cor}

\begin{proof}
By \cite[Theorem 2.2]{Dascha}, $\lambda(\I)\geq\bb$ for each P-ideal $\I$. Hence, by Proposition \ref{lprop}, we have:
$$\bb\leq\lambda(\I)=\lambda(\J)\leq\add(\J\text{QN-space})\leq\non(\J\text{QN-space})\leq\non(\J\text{wQN-space}).$$
Theorem \ref{Supina1} and Lemma \ref{llem} give us $\non(\J\text{QN-space})\leq\non(\I\text{QN-space})=\bb$. This proves part (a). To show part (b) observe that $\J$ is $\leq_{K}$-below the same $\mathtt{F_\sigma}$ ideal as $\I$ (hence, $\non(\J\text{wQN-space})=\bb$ by Theorem \ref{main}). Indeed, this follows from transitivity of the Kat\v{e}tov order, as $\J\leq_{K}\I\otimes\emptyset\leq_{K}\I$ by Lemma \ref{llem}, and the fact that $\J$ and $\I\otimes\emptyset$ are isomorphic.
\end{proof}

\begin{rem}
Note that the equality $\add(\J\text{QN-space})=\bb$ in Corollary \ref{add-cor3} cannot be derived from Corollary \ref{add-cor2}, as the ideal $\J$ from Corollary \ref{add-cor3} is never a P-ideal. Indeed, it suffices to show that $\I\otimes\emptyset$ is not a P-ideal, for any ideal $\I$ on $\omega$, and this is witnessed by the sequence $(\{n\}\times\omega)\subseteq\I\otimes\emptyset$. What is more, it is easy to show that $\I\otimes\emptyset$ is tall if and only if $\I$ is tall. Hence, if $\I$ is a tall analytic P-ideal generated by an unbounded submeasure (this is the case for instance for each summable ideal or a simple density ideal which is not an Erd\H{o}s-Ulam ideal -- see Subsection \ref{submeasures}), then item (b) of Corollary \ref{add-cor3} gives us $\add(\J\text{QN-space})=\bb$ for a tall ideal $\J$ which is not a P-ideal.
\end{rem}

\subsection*{Acknowledgement}

We would like to thank the anonymous referee for many valuable remarks that allowed to improve this paper. In particular, the author is grateful for pointing out Lemma \ref{K-KB} (which allowed to strengthen item (b) of Theorem \ref{main}), Proposition \ref{recenzent-prop1} and ideas which led to improvement of some results from Section $4$.

\end{document}